\documentclass[11pt,twoside]{article}

\usepackage{titling}
\usepackage{enumitem}

\usepackage[top=4cm, bottom=3.5cm, left=3.8cm, right=3.8cm, columnsep=20pt]{geometry}
\usepackage{xcolor}
\definecolor{LinkColor}{RGB}{145,80,170}
\definecolor{CiteColor}{RGB}{230,125,65}
\definecolor{FileColor}{RGB}{75,170,130}
\usepackage{hyperref}
\hypersetup{
	colorlinks=true,     
	linkcolor=LinkColor,    
	citecolor=CiteColor,     
	filecolor=FileColor,      
	urlcolor=FileColor,       
}
\usepackage{fancyhdr}
\fancypagestyle{first}{
	\pagestyle{fancy}
	\fancyhead[C]{}
	\fancyhead[RO,LE]{}
	\fancyhead[LO,RE]{}
	\fancyfoot[C]{}

}

\usepackage{amsmath,tabu}
\usepackage{amsthm}
\usepackage{amssymb}
\usepackage{mathrsfs}
\usepackage{mathtools}
\usepackage{dsfont}
\usepackage{etoolbox}
\newcommand{\rom}[1]{\uppercase\expandafter{\romannumeral #1\relax}}

\usepackage{tabularx}
\usepackage{graphicx}
\usepackage{caption}
\usepackage{subcaption}
\usepackage{float}
\usepackage{booktabs}
\usepackage{wrapfig}
\usepackage{tikz}
\usepackage{tikz-cd}

\newtheoremstyle{lem}{}{}{\slshape}{}{\bfseries}{}{.5em}{}
\theoremstyle{lem}
\newtheorem{thm}{Theorem}
\newtheorem{lem}[thm]{Lemma}
\newtheorem{prop}[thm]{Proposition}
\newtheorem{cor}[thm]{Corollary}
\newtheorem*{defn*}{Definition}

\newtheoremstyle{rem}{}{}{\slshape}{}{\itshape}{}{.5em}{}
\theoremstyle{rem}
\newtheorem{rem}[thm]{Remark}

\newtheoremstyle{pr}{}{}{}{}{\scshape}{.}{.5em}{}
\theoremstyle{pr}
\newtheorem*{pr*}{Proof}
\AtEndEnvironment{pr*}{\null\hfill\qedsymbol}

\usepackage{newpxtext,newpxmath}
\usepackage[utf8]{inputenc}
\usepackage[T1]{fontenc}

\usepackage{lipsum}

\newcommand{\Z}{\mathds{Z}}

\newcommand{\R}{\mathds{R}}
\newcommand{\C}{\mathds{C}}
\newcommand{\D}{\mathds{D}}
\newcommand{\Id}{\mathds{1}}

\newcommand{\del}{\partial}

\newcommand{\nrg}[1]{\mathrm E(#1)}
\newcommand{\nrgL}[2]{\mathrm E_{#1}(#2)}
\newcommand{\Ham}{\mathrm{Ham}}
\newcommand{\LHam}{\mathcal L^{\Ham}}

\let\phi\varphi
\let\epsilon\varepsilon

\AtEndDocument{\bigskip{\footnotesize%
  \textsc{J.-P. Chass\'e: ETH Z\"urich, D-MATH, R\"amistrasse 101, 8092 Zurich, Switzerland} \par
  \textit{E-mail address}, \texttt{jeanphilippe.chasse@math.ethz.ch} \par
  \addvspace{\medskipamount}
  \textsc{R. Leclercq: Université Paris-Saclay, CNRS, Laboratoire de mathématiques d'Orsay, 91405, Orsay, France.} \par
  \textit{E-mail address}, \texttt{remi.leclercq@universite-paris-saclay.fr} \par
}}

\begin{document}

\title{A Hölder-type inequality for the Hausdorff distance between Lagrangians}
\author{Jean-Philippe Chassé\thanks{The first author was partially supported by the Swiss National Science Foundation (grant number  200021\_204107).} , Rémi Leclercq\thanks{The second author is partially supported by the ANR grant 21-CE40-0002 (CoSy).}}

\maketitle

\abstract{\noindent
  We prove a Hölder-type inequality (in the spirit of Joksimović--Seyfaddini \cite{JoksimovicSeyfaddini2022}) for the Hausdorff distance between Lagrangians with respect to the Lagrangian spectral distance or the Hofer--Chekanov distance. This inequality is established via methods developped by the first author \cite{Chasse2023,Chasse2022} in order to understand the symplectic geometry of certain collections of Lagrangians under metric constraints.
}

\section{Introduction}
\label{sec:introduction}

Let $(M,\omega)$ be a symplectic manifold with an $\omega$-compatible almost complex structure $J$. If $M$ is noncompact, we assume that $J$ is convex at infinity. We equip $M$ with the Riemannian metric $g=g_J=\omega(\cdot,J\cdot)$~---~we may assume it is complete and geometrically bounded (cf.~\cite{CieliebakGinzburgKerman2004}). \par

\subsection{Main result}
\label{sec:main-result}

On one hand, Joksimović and Seyfaddini \cite{JoksimovicSeyfaddini2022} proved a Hölder-type inequality for the $C^{0}$ distance on Hamiltonian diffeomorphism groups and deduced interesting applications to Anosov--Katok pseudo-rotations.

Namely, the inequality is the following:
  \begin{align} \label{eqn:dushan_sobhan}
    d_{C^0}(\Id,\phi)\leq C\sqrt{\gamma(\phi)} \, ||d\phi||,
  \end{align}
where $||d\phi||:=\sup \{ |d\phi_x|^\mathrm{op} \,|\, x \in M \}$.
This inequality holds for any Hamiltonian diffeomorphism $\phi$ of a closed symplectic manifold for which one can define Floer spectral invariants. These invariants and their properties are reviewed in Section \ref{sec:preliminaries}. They induce, on the Hamiltonian diffeomorphism group, the \emph{spectral pseudonorm} $\gamma$ which appears in the inequality. The constant $C$ only depends on the choice of a Riemannian metric on the ambient manifold.

On the other hand, the first author \cite{Chasse2023,Chasse2022} initiated the study of the symplectic geometry of certain sets of Lagrangians under (Riemannian) metric constraints, such as the Hofer geometry of Hamiltonian isotopic Lagrangians with uniformly bounded curvature. Hölder inequalities on such sets between the Hausdorff distance and a large class of metrics were also obtained. This class of metrics contains, in particular, the \emph{Lagrangian spectral distance}, also defined via spectral invariants and denoted $\gamma$ as well. This version of the spectral distance was defined for weakly exact Lagrangians in \cite{Leclercq2008} and for monotone Lagrangians with nonvanishing fundamental in \cite{KislevShelukhin2022}. 

\medskip
The upshot of this note is a Hölder-type inequality for the Hausdorff distance $\delta_H$ between Lagrangians in the spirit of Joksimović and Seyfaddini's inequality, whose proof is based on the methods of \cite{Chasse2023}.

\begin{thm} \label{thm:main}
  Let $L$ and $L'$ be Hamiltonian isotopic, closed, connected Lagrangian submanifolds of $M$.
  Suppose that $L$~---~and thus $L'$~---~is either weakly exact or monotone with $N_{L} \geq 2$ and has nonvanishing quantum homology. Let $\psi$ be a symplectomorphism of $M$ such that $\psi(L)=L'$. There exist constants $\delta=\delta(M,J,L)>0$ and $C=C(M,J,L)>0$ such that whenever $\gamma(L,L')< \delta$, then
	\begin{align} \label{eqn:main}
	\delta_H(L,L')\leq C\sqrt{\gamma(L,L')}\ ||d\psi|| \,. 
	\end{align}
	Furthermore, when $M$ is compact, we may take $\delta=+\infty$.
\end{thm}

This obviously yields the nondegeneracy of $\gamma$.
\begin{cor}
  Let $L$ be a Lagrangian as in Theorem~\ref{thm:main}. Then, the Lagrangian pseudodistance $\gamma$ is nondegenerate on the set of Lagrangians which are Hamiltonian isotopic to $L$.
\end{cor}
Thus, this provides a third proof of this result, after Kawasaki's proof \cite{kawasaki2018function} via Poisson bracket invariants \textit{à la} Polterovich--Rosen \cite{MR3241729} and Kislev and Shelukhin's proof \cite{KislevShelukhin2022} via energy-capacity inequalities. However, through the use of the methods of \cite{Chasse2023} in Section~\ref{sec:proof-theorem-1} or with the more direct approach of the alternative proof in Section~\ref{sec:joks-seyf-appr}, the proof does ultimately relies on the same existence result for certain $J$-holomorphic curves as Kislev and Shelukhin~\cite{KislevShelukhin2022}. The innovation here is how we estimate the area of those $J$-holomorphic curves.

\medskip
Before describing in more details how this result relates to the aforementioned previous works, let us make several quick remarks.

%

\begin{rem}[Hofer's geometry]\label{remark:Hofer-distance}
  It is well known that $\gamma$ is bounded from above by the Hofer--Chekanov distance~---~see the properties of the spectral distance in Section \ref{sec:preliminaries}. Hence, Theorem~\ref{thm:main} also holds when $\gamma$ is replaced by the Hofer--Chekanov distance.
\end{rem}

\begin{rem}[Symplectomorphisms]\label{remark:symplectomorphism-diffeomorphism}
  Let us emphasize the fact that Theorem \ref{thm:main} holds for any~---~even noncompactly-supported~---~\emph{symplectomorphism} $\psi$; $L$ and $L'$ are required to be Hamiltonian diffeomorphic only for $\gamma(L,L')$ to be defined. \par
  
  That this is a nontrivial improvement, one can see with a simple example. For example, take $f(x)=\sin(2\pi x)$ and consider $L=\operatorname{graph}(f')$ and $L'=\operatorname{graph}(-f')$ in $T^*S^1$. If we equip $T^*S^1$ with the flat metric, then the symplectomorphism $\psi(x,y)=(-x,-y)$ sends $L$ to $L'$ and is also an isometry, i.e.\ $||d\psi||=1$. However, the only nontrivial isometries of the flat cylinder which are in the connected component of the identity (in $\operatorname{Diff}(T^*S^1)$) are translations, none of which are Hamiltonian diffeomorphisms¨. Therefore, every Hamiltonian diffeomorphism $\phi$ sending $L$ to $L'$ must have $||d\phi||>1$.
\end{rem}

%

%

\begin{rem}[Variant with the norm of the inverse diffeomorphism]\label{remark:inverse-norm}
  When $M$ is compact, there is also an inequality involving $||d\psi^{-1}||$:
\begin{align} \label{eqn:main_inverse}
	\delta_H(L,L')\leq C\sqrt{\gamma(L,L')}\ ||d\psi^{-1}||^2.
\end{align}
This variant of inequality \eqref{eqn:main} will be proved in Section~\ref{sec:an-inequality-with-inverse-norm}.
\end{rem}

\subsection{Main techniques and relations to previous work}
\label{sec:main-techniques}

Theorem \ref{thm:main} is a specialization of the first author's inequality from~\cite{Chasse2023}, which we now recall. For any metric $D$ in a large class of metrics, said of Chekanov type and which includes $\gamma$, if $D(L,L')<\delta=\delta(g,g|_L,g|_{L'})$, then
  \begin{align} \label{eqn:riem_bounds}
    \delta_H(L,L')\leq C(g,g|_L,g|_{L'})\sqrt{D(L,L')} \,.
  \end{align}
By the above notation, we mean that $\delta$ and $C$ depend only on Riemannian bounds of $M$, $L$, and $L'$, e.g.\ the sectional curvature of the first and the $L^\infty$-norm of the second fundamental form of the two latter. The improvement in this note is that we get rid of the dependance of $C$ on metric invariants of $L'$ at the price of an extra $||d\psi||$ term. \par
		
Note that the first author (Lemma 5 in \cite{Chasse2022}) partially improved (\ref{eqn:riem_bounds}) to
	\begin{align} \label{eqn:half_ineq}
		s(L;L')\leq C(g,g|_L)\sqrt{\gamma(L,L')}
	\end{align}
	whenever $\gamma(L,L')·<\delta=\delta(g,g|_L)$, where
	\begin{align}\label{eq:def-of-s-and-hausdorff}
		s(L;L'):=\sup_{x\in L} d_M(x,L')=\sup_{x\in L}\inf_{y\in L'} d_M(x,y) \,.
	\end{align}
Since $\delta_H(L,L')=\max\{s(L;L'),s(L';L)\}$, the left-hand side in (\ref{eqn:half_ineq}) is in general smaller than the one in (\ref{eqn:riem_bounds}). It is through this inequality that Theorem~\ref{thm:main} is proved (see Section \ref{sec:proof-theorem-1}). \par

\subsection{Relations to Joksimovi{\'c} and Seyfaddini's inequality}
\label{sec:relat-JS}

Theorem \ref{thm:main} is a Lagrangian generalization of Joksimovi{\'c} and Seyfaddini's aforementioned inequality \eqref{eqn:dushan_sobhan} for Hamiltonian diffeomorphisms $\phi$ on closed symplectic manifolds.

Note that their inequality directly implies
\begin{align*}
  \delta_H(L,L')\leq\inf_{\phi(L)=L'} d_{C^0}(\Id,\phi)\leq C\inf_{\phi(L)=L'} \Big( \sqrt{\gamma(\phi)}\ ||d\phi|| \Big).
\end{align*}
However, in general, the inequality
\begin{align*}
  \inf_{\phi(L)=L'}\sqrt{\gamma(\phi)}\ ||d\phi||\geq \inf_{\phi(L)=L'}\sqrt{\gamma(\phi)}\,\cdot\!\!\inf_{\phi(L)=L'} ||d\phi||=\sqrt{\gamma(L,L')}\,\cdot\!\! \inf_{\phi(L)=L'}||d\phi||
\end{align*}
is strict. Therefore, our inequality gives a better bound in the Lagrangian case, even when $\phi$ is a Hamiltonian diffeomorphism\footnote{Recall indeed that inequality \eqref{eqn:main} also holds when $\phi$ is a non-Hamiltonian symplectomorphism.}.


\medskip
One notable exception to this is when $L$ is the diagonal in $M\times M$, and $L'$ is the graph of $\phi$. Then, by work of the second author and Zapolsky~\cite{Leclercq2008,LeclercqZapolsky2018}, we know that $\gamma(L,L')=\gamma(\phi)$, so that equality follows. The constant we get here is however hard to compare to theirs.

On the other hand, we present below a different proof of a variant of (\ref{eqn:main}), based on the method from \cite{JoksimovicSeyfaddini2022} which gives a less natural, but more easily comparable, constant see Section \ref{sec:joks-seyf-appr}.

\subsection*{Organization and acknowledgements}
\label{sec:organ-this-note}

After reviewing necessary preliminaries in Section \ref{sec:preliminaries}, we prove Theorem \ref{thm:main} in Section \ref{sec:proof-theorem-1}. Finally, Section \ref{sec:other-stuff-addit} presents the proofs of two inequalities similar to \eqref{eqn:main}, the first one based on Joksimović and Seyfaddini's method in Section \ref{sec:joks-seyf-appr}, the second one involving the inverse norm of $\psi$ as mentioned in Remark \ref{remark:inverse-norm}; see Section \ref{sec:an-inequality-with-inverse-norm}.

\medskip

The main lines of this project were drawn during a stay of the first author at the Institut Mathématique d'Orsay. We thank the Laboratoire Mathématique d'Orsay for making that stay possible. 


%
%

\section{Preliminaries}
\label{sec:preliminaries}

We fix a symplectic manifold $(M,\omega)$ and consider different types of Lagrangian submanifolds. They are characterized by two functions defined on the second homotopy group of $M$ relative to $L$, i.e. the symplectic area and the Maslov class of disks in $M$ with boundary in $L$:
\begin{equation*}
  \omega_{L} : \pi_{2}(M,L) \to \R \quad\mbox{ and }\quad \mu_{L} : \pi_{2}(M,L) \to \Z \,.
\end{equation*}
A Lagrangian submanifold $L$ is called \emph{weakly exact} if $\omega_{L}$ and $\mu_{L}$ vanish identically. Otherwise,  $L$ is called (positively) monotone whenever there exists a positive constant $\kappa_{L}>0$ such that $\omega_{L} = \kappa_{L} \cdot \mu_{L}$. In that case, $\kappa_{L}$ is called the monotonicity constant of $L$. 

When $L$ is monotone, we define its \emph{minimal Maslov number} $N_{L}$ to be the positive generator of $\langle \mu_{L} , \pi_{2}(M,L) \rangle = N_{L}\,\Z$, and we require $N_{L} \geq 2$. 

\medskip
In what follows, we fix a Lagrangian as above and consider the set $\LHam(L)$ of all Lagrangian submanifolds which are Hamiltonian isotopic to $L$. We now recall how the two metrics on $\LHam(L)$ of interest in this note are defined.

\subsection{The Hofer--Chekanov distance}
\label{sec:hofer-distance}

The Hofer norm was introduced by Hofer \cite{MR1059642} on Hamiltonian diffeomorphism groups and extended as a distance to sets of the type $\LHam(L)$ by Chekanov \cite{MR1774099}.

First define the energy of a Hamiltonian function $H:[0,1]\times M\to\R$ as its $L^{(1,\infty)}$-norm:
\begin{equation}\label{eq:Hofer-nrg-H}
  \nrg H = \int_{0}^{1} \left(\max_{M}H_{t} - \min_{M}H_{t}\right) \, dt,
\end{equation}
where $H_t:=H(t,\cdot)$. Then, define the Hofer norm of a Hamiltonian diffeomorphism as
\begin{equation*}
  \| \phi \|_{\mathrm{Hof}} = \inf \left\{ \nrg H \,\middle|\, \phi_{H}^{1} = \phi \right\} \,.
\end{equation*}
Here, $\{\phi_H^t\}_{t\in [0,1]}$ is the Hamiltonian flow of $H$, i.e.\ $\phi_H^0=\Id_M$ and $\frac{d}{dt}\phi_H^t=X_H^t\circ\phi_H^t$, where $X_H^t$ is the unique time-dependent vector field of $M$ such that $\iota(X_H^t)\omega=-dH_t$.

Hofer's norm then yields a distance on $\LHam (L)$ by setting
\begin{equation*}
  d_{\mathrm{Hof}}(L,L') = \inf \left\{\|\phi\|_{\mathrm{Hof}} \,\middle|\, \phi(L) = L'\right\} = \inf \left\{ \nrg H \,\middle|\, \phi_{H}^{1}(L) = L' \right\}
\end{equation*}
for any $L' \in \LHam (L)$.

\begin{rem}
  As noted by Usher~\cite{MR3507263}, replacing the Hofer energy $\nrg H$ in the above expression by the smaller quantity $\nrgL{L}{H}$, defined as in \eqref{eq:Hofer-nrg-H} but with oscillations taken only on $L$ rather than on the whole ambient manifold $M$, yields the same distance.
\end{rem}

\subsection{The Lagrangian spectral distance}
\label{sec:lagr-spectr-dist}

This distance is based on the theory of spectral invariants initiated by Viterbo \cite{MR1157321} via  generating functions and adapted to Floer homology theories by Schwarz \cite{MR1755825} and Oh \cite{MR2103018} in the case of Hamiltonian diffeomorphism groups. The Lagrangian version which is of interest to us here was developed by the second author~\cite{Leclercq2008} in the weakly exact setting and by Zapolsky and the second author~\cite{LeclercqZapolsky2018} in the monotone case~---~see also work by Fukaya, Oh, Ohta, and Ono \cite{MR3986938}, which is based on more advanced techniques such as virtual fundamental cycles and Kuranishi structures.

\paragraph{Lagrangian spectral invariants}

The Lagrangian spectral invariants $\ell(\alpha;H)$ associated to $L$ are defined for any nonzero quantum homology class $\alpha \in \mathrm{QH}_{*}(L)$~---~see~\cite{BiranCornea2009} for the construction of this homology. Since the Lagrangian spectral \emph{distance} only relies on spectral invariants corresponding to the quantum fundamental class of $L$, we do not review the construction of the quantum homology of a Lagrangian, nor define spectral invariants in full generality.
Instead, we assume that the quantum fundamental class of $L$, denoted $[L]$, is nontrivial and only present the properties of $\ell_{+}:=\ell([L], \,\cdot\,)$.

\medskip
The function
$\ell_{+} : C^{0}([0,1]\times M) \rightarrow \mathbb R$
satisfies the following properties.
\begin{enumerate}
  \item \textsc{Continuity.} For any Hamiltonians $H$ and $K$, we have that
    \begin{equation*}
      \int_{0}^{1}\min_{M}(K_{t}-H_{t}) \,dt \leq \ell_{+}(K) - \ell_{+}(H) \leq \int_{0}^{1}\max_{M} (K_{t}-H_{t}) \,dt \,.
    \end{equation*}
  \item \textsc{Triangle inequality.} For all $H$ and $K$, $\ell_{+}(H\sharp K) \leq \ell_{+}(H) + \ell_{+}(K)$.
  \item \textsc{Lagrangian control.} If $H_{t}|_{L} = c(t) \in \mathbb R$ (resp. $\leq$, $\geq$) for all $t$, then
    \begin{equation*}
      \ell_{+}(H) = \int_{0}^{1}c(t) \, dt \qquad \mbox{ (resp. $\leq$, $\geq$).}
    \end{equation*}
  \item \textsc{Non-negativity.} For all $H$, $\ell_{+}(H) + \ell_{+}(\overline H) \geq 0$.
  \item \textsc{Homotopy invariance.} If $H$ is normalized, $\ell_{+}(H)$ only depends on the homotopy class relative to endpoints of the isotopy $\{\phi_{H}^{t}\}_{t \in [0,1]}$, i.e. the class $[\{\phi_{H}^{t}\}_{t \in [0,1]}] \in \widetilde\Ham(M,\omega)$.
  \item \textsc{Symplectic invariance.} For all $H$ and all $\psi \in \mathrm{Symp}(M,\omega)$, $\ell_{+}(H) = \ell'_{+}(H \circ\psi^{-1})$.
  \end{enumerate}

\medskip
Let us make a few comments about these properties and the notation used above.
  \begin{itemize}
  \item In Properties 2 and 4 respectively, $H\sharp K$ denotes the
    Hamiltonian function
    $H_{t}(x) + K\big((\phi_{H}^{t})^{-1}(x)\big)$ which generates the
    isotopy $\{\phi_{H}^{t}\phi_{K}^{t}\}_{t \in [0,1]}$, and
    $\overline H$ is the Hamiltonian function
    $\overline H_{t}(x) = - H_{t}\big((\phi_{H}^{t})^{-1}(x)\big)$ which
    generates $\big\{(\phi_{H}^{t})^{-1}\big\}_{t \in [0,1]}$.
    Properties 1 to 4 are part of Theorem 3 in \cite{LeclercqZapolsky2018}.
  \item Property 3 directly implies that for all $H$,
  \begin{equation*}
  	\int_{0}^{1} \min_{L} H_{t} \,dt \leq \ell_{+}(H) \leq \int_{0}^{1} \max_{L} H_{t} \,dt \,.
  \end{equation*}

  \item In Property 5, the \emph{normalization} refers to the fact that for all $t$, $\int_{0}^{1} H_{t} \omega^{n}=0$. This property appears as Proposition 4 in \cite{LeclercqZapolsky2018}.
    
  \item Finally, concerning Property 6, note that any symplectomorphism $\psi$ induces an isomorphism $\psi_{*} : \mathrm{QH}_{*}(L) \to \mathrm{QH}_{*}(L')$ with $L' = \psi(L)$. The fundamental class of $L$ is mapped to that of $L'$ through this action (up to possible multiplication by a unit of the coefficient field). The notation $\ell'_{+}$ denotes the Lagrangian spectral invariant associated to $L'$ (and its fundamental class). 
Now, \textsc{Symplectic invariance} only expresses the fact that spectral invariants agree with the action of $\psi$ by conjugation on the Hamiltonian diffeomorphism group: for any Hamiltonian function $H$, $\phi_{H\circ \psi}^{t} = \psi^{-1}\phi_{H}^{t}\psi$.
This result is part of Theorem 35 in \cite{LeclercqZapolsky2018}.
\end{itemize}

\paragraph{The Lagrangian spectral distance}

The properties of $\ell_{+}$ above show that not only $\ell_{+}$ defines a function on $\widetilde\Ham(M,\omega)$ with similar properties (see Theorem 41 in \cite{LeclercqZapolsky2018}), but also a pseudodistance on $\LHam (L)$.

Indeed, following \cite{KislevShelukhin2022}, first define the length of a Hamiltonian isotopy $\{\phi_{H}^{t}\}_{t \in [0,1]}$ by $\gamma_{L}(H) = \ell_{+}(H)+\ell_{+}(\overline H)$, then take the infimum over all Hamiltonian isotopies which map $L$ to $L'$:
\begin{equation*}
  \gamma(L,L') = \inf \{ \gamma_{L}(H) \,|\, \phi_{H}^{1}(L) = L'  \} \,.
\end{equation*}

The \textsc{Non-negativity} property of $\ell_{+}$ ensures that $\gamma(L,\cdot)$ takes non-negative values. \textsc{Symplectic invariance} ensures that for any symplectomorphism $\psi$ and any Lagrangian $L' \in \LHam(L)$, $\gamma(L,L') = \gamma(\psi(L),\psi(L'))$. Combined with \textsc{Triangle inequality}, this shows that for all $L'$ and $L''$ in $\LHam(L)$,
\begin{equation*}
  \gamma(L,L'') \leq \gamma(L,L') + \gamma(L',L'') \,.
\end{equation*}
Finally, note that if $L'=\phi^1_H(L)$, then $L=\phi^1_{\overline{H}}(L')$ and
\begin{align*}
	\gamma_{L'}(\overline{H})
	&=\ell'_{+}(\overline{H})+\ell'_{+}(H) \\
	&=\ell_{+}(\overline{H}\circ\phi^1_H)+\ell_{+}(H\circ\phi^1_H) \\
	&=\ell_{+}(\overline{H\circ\phi^1_H})+\ell_{+}(H\circ\phi^1_H) \\
	&=\gamma_{L}(H\circ\phi^1_H),
\end{align*}
where the second line follows from \textsc{Symplectic invariance}, whilst the third one is a direct computation using the fact that $\phi^1_{H\circ\phi^1_H}=(\phi^1_H)^{-1}\phi^1_H\phi^1_H=\phi^1_H$. Since $\gamma(L',L)$ is defined by taking the infimum over all possible Hamiltonians whose diffeomorphism sends $L'$ to $L$, this implies symmetry for $\gamma$. This justifies the following definition.

\begin{defn*}
  Let $L$ be a weakly exact Lagrangian or a monotone Lagrangian with $N_{L}\geq 2$ and nonzero quantum fundamental class. The \emph{Lagrangian spectral distance} between $L_{0}$ and $L_{1} \in \LHam(L)$ is $\gamma(L_0,L_1)$. 
\end{defn*}

The fact that this actually defines a \emph{nondegenerate} distance is, as usual, the ``hard'' part. This was proven fairly simultaneously in \cite{kawasaki2018function} (via Poisson bracket invariants) and \cite{KislevShelukhin2022} (via energy-capacity inequality). This is also a consequence of the main result of the present note.

Finally, let us emphasize the fact that the \textsc{Continuity} property of $\ell_{+}$ obviously yields the well-known fact that 
\begin{equation*}
  \mbox{for all $L' \in \LHam (L)$}, \quad \gamma(L,L') \leq d_{\mathrm{Hof}} (L,L')\,.
\end{equation*}

\section{Proof of Theorem~\ref{thm:main}}
\label{sec:proof-theorem-1}

Fix a Lagrangian submanifold $L$ which satisfies the assumptions of Theorem~\ref{thm:main}. Let $L' = \phi_{H}^{1}(L) \in \LHam(L)$ for some Hamiltonian function $H$, and let $\psi \in \mathrm{Symp}(M,\omega)$ be such that $L' = \psi(L)$. Notice that $\psi^{-1}(L) \in\LHam(L)$ since the Hamiltonian function $H\circ \psi$ generates the isotopy $\{\psi^{-1}\phi_{H}^{t}\psi\}$ which maps $\psi^{-1}(L)$ to $L$ at time 1.

\medskip
Recall from \eqref{eq:def-of-s-and-hausdorff} that the Hausdorff distance between $L$ and $L'$ is defined as $\delta_{H} = \max \{s(L;L'),s(L';L)\}$, where $s(A;B)$ is the supremum of the distance to $B$ of a point in $A$.

From (\ref{eqn:half_ineq}), i.e. Lemma~5 of \cite{Chasse2022}, we get some constants $\delta=\delta(g,g|_L)>0$ and $C=C(g,g|_L)>0$ such that
\begin{align*}
s(L;\psi(L)) \leq C\sqrt{\gamma(L,\psi(L))}
\quad\mbox{ and }\quad
s(L;\psi^{-1}(L)) \leq C\sqrt{\gamma(L,\psi^{-1}(L))}
\end{align*}
whenever $\gamma(L,\psi(L))$ and $\gamma(L,\psi^{-1}(L))$ are smaller than $\delta$. \par

Let $\ell(c)$ denote the length of a smooth path $c:[0,1]\to M$. Then,
\begin{align*}
s(\psi(L);L)
&= \max_{y\in\psi(L)}\min_{\substack{c(0)=y \\ c(1)\in L}} \ell(c) \\
&= \max_{x\in L}\min_{\substack{c(0)=x \\ c(1)\in \psi^{-1}(L)}} \ell(\psi\circ c) \\
&\leq ||d\psi|| \max_{x\in L}\min_{\substack{c(0)=x \\ c(1)\in \psi^{-1}(L)}} \ell(c) \\
&= ||d\psi||\ s(L;\psi^{-1}(L)) \,.
\end{align*}
From this, we immediately get that
\begin{align} \label{eqn:pre-main}
\delta_H(L,L')
&\leq \max\big\{s(L;\psi(L)),||d\psi||\ s(L;\psi^{-1}(L))\big\} \nonumber\\
&\leq C||d\psi||\max\left\{\sqrt{\gamma(L,\psi(L))},\sqrt{\gamma(L,\psi^{-1}(L))}\right\}
\end{align}
since $||d\psi||\geq 1$ for any symplectomorphism $\psi$. Indeed, a symplectic matrix must always have an eigenvalue with absolute value at least 1. \par

By \textsc{Symplectic invariance}, we know that $\gamma(L,\psi^{-1}(L)) = \gamma(\psi(L),L)$, and 
%
(\ref{eqn:pre-main}) gives us the expected inequality (\ref{eqn:main}):
\begin{equation*}
  \delta_{H}(L,L') \leq C\|d\psi\| \sqrt{\gamma(L,L')}
\end{equation*}
under the condition $\gamma(L,L')<\delta$. \par

To get rid of this condition when $M$ is compact, we use Joksimovi{\'c} and Seyfaddini's trick~\cite{JoksimovicSeyfaddini2022}: take $C$ large enough so that
\begin{align*}
	C\geq\frac{\mathrm{Diam}(M)}{\sqrt{\delta}}.
\end{align*} 
Then, if $\gamma(L,L')\geq\delta$, we trivially get
\begin{align*}
	C\sqrt{\gamma(L,L')}\ ||d\psi||\geq \mathrm{Diam}(M)\geq \delta_H(L,L'),
\end{align*}
since $||d\psi||\geq 1$. Here, we have made use of the fact that the distance between two closed subsets of $M$ is at most the diameter of $M$. \par

\medskip
This ends the proof of Theorem \ref{thm:main}. \hfill$\square$

\section{Alternative versions of inequality (\ref{eqn:main})}
\label{sec:other-stuff-addit}

We conclude with two alternative versions of inequality (\ref{eqn:main}): the first one is established by adapting to the Lagrangian setting Joksimović and Seyfaddini's proof from \cite{JoksimovicSeyfaddini2022}, and the other one by using methods explored by Chassé and {leading} to inequality \eqref{eqn:riem_bounds} from \cite{Chasse2023}, rather than using directly inequality \eqref{eqn:half_ineq} from \cite{Chasse2022}.

\subsection{Joksimovi{\'c} and Seyfaddini's approach}\label{sec:joks-seyf-appr}
We could have adapted Joksimovi{\'c} and Seyfaddini's~\cite{JoksimovicSeyfaddini2022} proof of (\ref{eqn:dushan_sobhan}) to the Lagrangian context to get an analogous inequality. We give here the broad idea on how such an inequality is proven. \par
		
For each $x\in L$, take a Darboux chart $\psi_x:U_x\to\R^{2n}$ sending $L\cap U_x$ to $\R^n\times\{0\}$. Take also compact neighborhoods $K_x$ and $K'_x$ of $x$ in $M$ such that
\begin{align*}
	K_x\subseteq \mathrm{int}(K'_x)\subseteq K'_x\subseteq U_x.
\end{align*}
By compactness of $L$, we may take a finite subset $\{\psi_i\}_{1\leq i\leq k}$ of these charts, so that $\{\mathrm{int}(K'_i)\}_{1\leq i\leq k}$ still covers $L$. Then, setting
\begin{align*}
	\epsilon:=\min_{1\leq i\leq k} \min_{\substack{x\in\del K_i\\ x'\in\del K'_i}} d(x,x') 
	\quad\mbox{ and }\quad
	A:=\max_{1\leq i\leq k}\left|\left|\left.d\psi_i^{-1}\right|_{\psi_i(K'_i)}\right|\right|,
\end{align*}
we get the inclusion
\begin{align*}
	\psi_i^{-1}(B^{2n}_r(\psi_i(x)))\subseteq B_{Ar}(x)
\end{align*}
for $r=2\sqrt{\frac{\gamma(L,L')}{\pi}}$ if $\gamma(L,L')<\delta=\frac{\pi\epsilon^2}{4A^2}$. Here, $B^{2n}$ denotes the Euclidean ball in $\R^{2n}$, whilst $B$ is the metric ball in $M$. But then, if $Ar< d(x,L')$, the map $\psi_i^{-1}|_{B^{2n}_r(\psi_i(x))}$ would be a symplectic embedding of a ball of radius $r$ with real part along $L$ not crossing $L'$, so that
\begin{align*}
	\gamma(L,L')\geq \frac{\pi}{2}r^2= 2\gamma(L,L')
\end{align*}
by the proof of Theorem~E of \cite{KislevShelukhin2022}, which is of course a contradiction. Therefore, we must have
\begin{align*}
	d(x,L')\leq 2A\sqrt{\frac{\gamma(L,L')}{\pi}}.
\end{align*}
This gives an inequality analogous to (\ref{eqn:half_ineq})~---~but with a constant depending on local charts~---~by taking the maximum over all $x\in L$. \par
		
\subsection{An inequality with the inverse norm}\label{sec:an-inequality-with-inverse-norm}
When $M$ is compact, there is also an inequality with $\|d\psi^{-1}\|$:
\begin{align} \label{eqn:main_inverse}
	\delta_H(L,L')\leq C\sqrt{\gamma(L,L')}\ \|d\psi^{-1}\|^2.
\end{align}

The proof of (\ref{eqn:main_inverse}) follows the scheme of the proof of (\ref{eqn:riem_bounds}) appearing in \cite{Chasse2023}. We thus recall the idea of said proof.
\begin{enumerate}[label=(\arabic*)]
	\item From the proof of Theorem~E of \cite{KislevShelukhin2022}, we know that there exist, for any $x\in L$ and any $x'\in L'$, $J$-holomorphic strips $u_x$ and $u_{x'}$ with boundary along $L$ and $L'$~---~modulo arbitrarily small Hamiltonian perturbations~---~and passing through $x$ and $x'$, respectively. Furthermore, their area is bounded from above by $2\gamma(L,L')$. \par
	
	\item Using a version of the monotonicity lemma, we get that
	\begin{align*}
	\omega(u_x)\geq A(g,g|_L)r^2
	\end{align*}
	if the closed metric ball $B_r(x)$ does not intersect $L'$ and $r$ is smaller than some $\delta=\delta(g,g|_L)>0$. There is an analogous result for $u_{x'}$ and $L'$. \par
	In particular, if $\gamma(L,L')$ is small enough, the inequality holds for all $r<d_M(x,L')$. Therefore, it holds for $r=d_M(x,L')$. \par
	
	\item Taking the supremum over all $x\in L$ of the inequalities for $L$, we essentially get (\ref{eqn:half_ineq}). Taking the supremum over all $x'\in L'$ of the inequalities for $L'$ gives an analogous inequality for the pair $(L',L)$. Taking the maximum of these two inequalities, we get (\ref{eqn:riem_bounds}). \par
\end{enumerate}
We thus see that the dependence of $C$ in (\ref{eqn:riem_bounds}) on metric invariants of $L'$ comes from the constant $A$ in Step~2. Therefore, proving Theorem~\ref{thm:main} reduces to proving the following proposition.

\begin{prop} \label{prop:monotonicity}
  There exist constants $\delta$ and $A$ depending only on metric invariants of $M$ and $L$ with the following property.
  
  Let $L' \in \LHam (L)$ and let $\psi$ be a symplectomorphism such that $\psi(L)=L'$.
  Let $\Sigma$ be a compact Riemann surface with boundary $\del\Sigma$ with corners. Consider a nonconstant $J$-holomorphic curve $u':(\Sigma,\del\Sigma)\to (B_r(x'),\del B_r(x')\cup L')$ for some $x'\in L'$ and $r\leq \frac{\delta}{\|d\psi^{-1}\|}$ such that $x'\in u'(\Sigma)$. Suppose that $u'$ sends the corners of $\Sigma$ to $\del B_r(x')\cap L'$. Then, 
	\begin{align*}
	\omega(u')\geq \frac{A}{\|d\psi^{-1}\|^2}r^2 \,.
	\end{align*}
\end{prop}

Indeed, Proposition 2.1 of~\cite{Chasse2023}, Proposition~\ref{prop:monotonicity}, and Step~(1) above yield
\begin{align*}
\min\left\{\delta,d_M(x,L')\right\}
&\leq C\sqrt{\gamma(L,L')}
\intertext{for all $x\in L$ and}
\min\left\{\frac{\delta}{\|d\psi^{-1}\|},d_M(x',L)\right\}
&\leq C\sqrt{\gamma(L,L')}\|d\psi^{-1}\|
\end{align*}
for all $x'\in L'$ (with $C=\frac{1}{\sqrt{2A}}$). In particular, if we suppose that $\gamma(L,L')<C^{-2}\delta^2\|d\psi^{-1}\|^{-4}\leq C^{-2}\delta^2$, we get that
\begin{align*}
d_M(x,L')
&\leq C\sqrt{\gamma(L,L')}
\intertext{for all $x\in L$ and}
d_M(x',L)
&\leq C\sqrt{\gamma(L,L')}\|d\psi^{-1}\| 
\end{align*}
for all $x'\in L'$. Taking the maximum over all $x$ and all $x'$, we get $\delta_H(L,L')\leq C\sqrt{\gamma(L,L')}\max\{1,\|d\psi^{-1}\|\}$ as long as $\gamma(L,L')<C^{-2}\delta^2\|d\psi^{-1}\|^{-4}$. This yields (\ref{eqn:main_inverse})~---~with the additional $\gamma$-smallness assumption~---~since $\|d\psi^{-1}\|\geq 1$. \par

If $\gamma(L,L')\geq C^{-2}\delta^2\|d\psi^{-1}\|^{-4}$, take $C'\geq C\delta^{-1}\operatorname{Diam}(M)$, so that
\begin{align*}
C'\sqrt{\gamma(L,L')}\|d\psi^{-1}\|^2\geq \operatorname{Diam}(M)\geq \delta_H(L,L'),
\end{align*}
which gives the desired result.

\medskip
Only Proposition~\ref{prop:monotonicity} is thus now left to prove. In order to do so, we first need a new version of the isoperimetric inequality. For an arc $\gamma':([0,\pi],\{0,\pi\})\to (M,L')$ whose image in contained in the metric ball $B_{\delta/\|d\psi^{-1}\|}(x')$ for some $x'\in L'$, set $a(\gamma')$ to be the symplectic area $\omega(u')$ of any map $u':\D\cap\{\operatorname{Im}z\geq 0\}\to M$ such that $u'(e^{i\theta})=\gamma'(\theta)$ and $u'(\D\cap\R)\subseteq L'$. Here, $\D$ is the unit disk in $\C$. 

First note that this definition is independent of the choice of extension $u'$. To see this, take
\begin{align}
	\delta=\min\left\{\epsilon,\frac{\epsilon}{2}r_\mathrm{inj}(L),\frac{\epsilon}{2}r_0,\frac{\pi}{4\sqrt{K_0}}\right\}.
\end{align}
if $L$ is $\epsilon$-tame (see~\cite{Chasse2023} for the definition) and $M$ has injectivity radius bounded away from zero by $r_0$ and sectional curvature takes values in $[-K_0,K_0]$. Here, $r_\mathrm{inj}(L)$ is the injectivity radius of $L$ with the Riemannian metric induced by $M$. Then, for all $z\in \D\cap\{\operatorname{Im}z\geq 0\}$, we have that
\begin{align*}
	d_M\left(\psi^{-1}(u(z)),\psi^{-1}(x')\right)\leq \|d\psi^{-1}\|\, d_M(u(z),x')\leq \delta,
\end{align*}
i.e.\ $u:=\psi^{-1}\circ u'$ has image in the metric ball $B_\delta(x)$ with $x:=\psi^{-1}(x)\in L$. Take two extensions $u'_0$ and $u'_1$ of an arc $\gamma'$ as above, and denote $\alpha'_i:=u'_i|_\R$ and $\alpha_i:=\psi^{-1}\circ\alpha'_i$. Then, $u_0\#\overline{u_1}$ is a disk whose boundary $\alpha_0\#\overline{\alpha_1}$ lies in $L$. Here, $\overline{f}(a+ib):=f(-a+ib)$ for any map $f:U\subseteq\C\to M$. But by $\epsilon$-tameness of $L$, $\alpha_0\#\overline{\alpha_1}$ must be a loop in a metric ball of $L$ (in the intrinsic metric) of radius $\frac{2\delta}{\epsilon}\leq r_\mathrm{inj}(L)$, and must thus be contractible in the same ball. This nullhomotopy extends to a homotopy in a metric ball of $M$ of radius $\frac{2\delta}{\epsilon}$ of $u_0\#\overline{u_1}$ to a topological sphere. Since $\frac{2\delta}{\epsilon}\leq r_0$, this topological sphere must be itself contractible, so that
\begin{align*}
	0=\omega(u_0\#\overline{u_1})=\omega(u_0)-\omega(u_1)=\omega(u'_0)-\omega(u'_1),
\end{align*}
where the last inequality follows from the fact that $\psi$ is a symplectomorphism. In other words, $a(\gamma')$ is indeed well defined.

We can now prove the following isoperimetric inequality.

\begin{lem} \label{lem:isoperimetric}
	There exist constants $\delta$ and $B$ depending only on metric invariants of $M$ and $L$ such that, for all arcs $\gamma':([0,\pi],\{0,\pi\})\to (M,L')$ with image in $B_{\delta/\|d\psi^{-1}\|}(x')$ for some $x'\in L'$, we have that
	\begin{align*}
		a(\gamma')\leq B\|d\psi^{-1}\|^2\ell(\gamma')^2.
	\end{align*}
\end{lem}
\begin{proof}
	As noted above, we know that $\psi^{-1}\circ\gamma$ has image in the metric ball $B_{\delta}(\psi^{-1}(x'))$. Therefore, by said Lemma~2.1 of \cite{Chasse2023}, we know that
	\begin{align*}
		a(\psi^{-1}\circ\gamma)\leq B(g,g|_L)\,\ell(\psi^{-1}\circ\gamma)^2.
	\end{align*}
	However, $a(\psi^{-1}\circ\gamma)=a(\gamma)$, since $\psi$ is a symplectomorphism, and $\ell(\psi^{-1}\circ\gamma)\leq \|d\psi^{-1}\|\ell(\gamma)$, which give the desired inequality.
\end{proof}

The proof of Proposition~\ref{prop:monotonicity} then follows the same scheme as the proof of Proposition~2.1 of \cite{Chasse2023}, excepts that we use Lemma~\ref{lem:isoperimetric} above~---~instead of Lemma~2.1 of \cite{Chasse2023}~---~to estimate the local action $a(\gamma')$ for arcs $\gamma'$ with boundary on $L'$.

\bibliographystyle{alpha}
\bibliography{ham-diffeo}

\end{document}